\documentclass[12pt,a4paper]{amsart}
\usepackage[utf8]{inputenc}
\usepackage{lmodern}
\usepackage{array,booktabs}
\usepackage{framed}
\usepackage{multicol}
\usepackage{amsmath}
\usepackage{amsfonts}
\usepackage{amssymb}
\usepackage{mathrsfs}
\usepackage{graphicx}
\usepackage{geometry}

\title[Riesz Decompositions ]{Riesz Decompositions for Schrödinger Operators on Graphs}

\author{Florian Fischer}
\address{Florian Fischer, Institute of Mathematics, University of Potsdam, Germany}
\email{florifis@uni-potsdam.de}
\author{Matthias Keller}
\address{Matthias Keller, Institute of Mathematics, University of Potsdam, Germany}
\email{matthias.keller@uni-potsdam.de}

\usepackage[hidelinks]{hyperref}
\hypersetup{%
	plainpages=false,%
	pdfauthor={Author(s)},%
	pdftitle={Title},%
	pdfsubject={Subject},%
	bookmarksnumbered=true,%
	colorlinks=true,
	citecolor=blue,
	filecolor=black,%
	linkcolor=blue,
	urlcolor=black,%
	pdfstartview=FitH%
}
\usepackage{pgffor}
\usepackage{enumitem}
\allowdisplaybreaks	
\setenumerate{label=(\alph*)}

\newtheorem{theorem}{Theorem}[section]
\newtheorem{lemma}[theorem]{Lemma}
\newtheorem{proposition}[theorem]{Proposition}
\newtheorem{corollary}[theorem]{Corollary}

\theoremstyle{definition}

\newtheorem{remark}[theorem]{Remark} 
\newtheorem{definition}[theorem]{Definition}

\foreach \x in {A,...,Z}{%
\expandafter\xdef\csname c\x\endcsname{\noexpand\ensuremath{\noexpand\mathcal{\x}}}
}
\foreach \x in {A,...,Z}{%
\expandafter\xdef\csname \x\x\endcsname{\noexpand\ensuremath{\noexpand\mathbb{\x}}}
}
\foreach \x in {A,...,Z}{%
\expandafter\xdef\csname s\x\endcsname{\noexpand\ensuremath{\noexpand\mathscr{\x}}}
}
\foreach \x in {A,...,Z}{%
\expandafter\xdef\csname b\x\endcsname{\noexpand\ensuremath{\noexpand\mathbf{\x}}}
}
\foreach \x in {a,...,z}{%
\expandafter\xdef\csname b\x\endcsname{\noexpand\ensuremath{\noexpand\mathbf{\x}}}
}
\foreach \x in {A,...,Z}{%
\expandafter\xdef\csname tilde\x\endcsname{\noexpand\ensuremath{\noexpand\widetilde{\x}}}
}
\foreach \x in {a,...,z}{%
\expandafter\xdef\csname tilde\x\endcsname{\noexpand\ensuremath{\noexpand\widetilde{\x}}}
}
\foreach \x in {A,...,Z}{%
\expandafter\xdef\csname hat\x\endcsname{\noexpand\ensuremath{\noexpand\widehat{\x}}}
}
\foreach \x in {a,...,z}{%
\expandafter\xdef\csname hat\x\endcsname{\noexpand\ensuremath{\noexpand\widehat{\x}}}
}
\foreach \x in {A,...,Z}{%
\expandafter\xdef\csname bar\x\endcsname{\noexpand\ensuremath{\noexpand\bar{\x}}}
}
\foreach \x in {a,...,z}{%
\expandafter\xdef\csname bar\x\endcsname{\noexpand\ensuremath{\noexpand\bar{\x}}}
}
\foreach \x in {A,...,Z}{%
\expandafter\xdef\csname btilde\x\endcsname{\noexpand\ensuremath{\noexpand\mathbf{\widetilde{\x}}}}
}
\foreach \x in {a,...,z}{%
\expandafter\xdef\csname btilde\x\endcsname{\noexpand\ensuremath{\noexpand\mathbf{\widetilde{\x}}}}
}
\foreach \x in {A,...,Z}{%
\expandafter\xdef\csname bhat\x\endcsname{\noexpand\ensuremath{\noexpand\mathbf{\widehat{\x}}}}
}
\foreach \x in {a,...,z}{%
\expandafter\xdef\csname bhat\x\endcsname{\noexpand\ensuremath{\noexpand\mathbf{\widehat{\x}}}}
}
\foreach \x in {A,...,Z}{%
\expandafter\xdef\csname d\x\endcsname{\noexpand\ensuremath{\noexpand\mathds{\x}}}
}

\newcommand{\abs}[1]{\left\lvert #1\right\rvert} 
\newcommand{\set}[1]{\left\{ #1\right\} }
\renewcommand{\epsilon}{\varepsilon}
\renewcommand{\phi}{\varphi}
\newcommand{\ph}{\varphi}

\DeclareMathOperator{\ee}{e}

\DeclareMathOperator{\supp}{supp}
\DeclareMathOperator{\dd}{d}
\DeclareMathOperator{\ghm}{ghm}

\newcommand{\sse}{\subseteq}

\newcommand{\eat}[1]{}

\newcommand{\Hmm}[1]{\leavevmode{\marginpar{\tiny%
			$\hbox to 0mm{\hspace*{-0.5mm}$\leftarrow$\hss}%
			\vcenter{\vrule depth 0.1mm height 0.1mm width \the\marginparwidth}%
			\hbox to 0mm{\hss$\rightarrow$\hspace*{-0.5mm}}$\\\relax\raggedright #1}}}
	
\begin{document}

\begin{abstract}We study superharmonic functions for Schrödinger operators on general weighted graphs. Specifically, we prove two decompositions which both go under the name Riesz decomposition in the literature. The first one decomposes a superharmonic function into a harmonic and a potential part. The second one decomposes a superharmonic function into a sum of superharmonic functions with  certain upper bounds given  by prescribed superharmonic functions. As application we show a Brelot type theorem.

\end{abstract}
\maketitle

\section{Introduction}
Schrödinger operators in the Euclidean space have been studied for a long time and a profound potential theory has been developed. On graphs  potential theory was mainly studied in the context of random walks. However, in recent years there is a rising interest in general Schrödinger operators which goes beyond the probabilistic framework. The analysis and spectral theory of these operators received enormous attention, see e.g. \cite{BGK15,BP,CTT11,GKS16,GMT14,GS11, KL12,KLSW17, KR16,KS17}. Especially, the study of Hardy inequalities \cite{G14,KPP16} relies on a profound understanding of (super)harmonic functions, see also \cite{Fitzsimmons,KePiPo1,Tak14,Tak16}.

A classical and fundamental tool to study superharmonic functions are  \emph{Riesz decompositions}.  In this paper we study two  of these decompositions for superharmonic functions of Schrödinger operators on graphs.  The first decompositions deals with   superharmonic functions which are bounded from below by a (sub)harmonic function. Then, the superharmonic function can be decomposed into  a harmonic and a potential part, see Theorem~\ref{thm:rieszdecompGHM}.
Such a decomposition is referred to in the literature as \emph{Riesz decomposition}, see \cite{AG, Helms, HelmsNeu}). The second decomposition considers a superharmonic function $ s $ which is smaller than the sum of two  superharmonic functions $ s_{1} $ and $ s_{2} $. Then, $ s $ can be decomposed into the sum of two superharmonic functions $ r_{1} $ and $ r_{2} $ such that  $ r_{1} \leq s_{1}$ and $ r_{2}\leq s_{2} $, see Theorem~\ref{thm:MS}. In the literature this is also referred to as Riesz decompostion, see \cite{Boboc,Hansen}, but also as Mokobodzki-Sibony decomposition, see \cite{HelmsNeu}, so, we will refer to it as \emph{Riesz-Mokobodzki-Sibony decomposition}.

In the context of random walks on graphs the first decomposition, the Riesz decompostion, is well known for non-negative superharmonic functions, see \cite{KSK,Soardi,WoessRandom,WoessMarkov}.
However, the  Schrödinger operators we study here do not have a probabilistic interpretation. Moreover, we wish to treat not only non-negative superharmonic functions but also superharmonic functions which are only bounded from below by a subharmonic functions. 
To this end, the probabilistic approach does not seem to work and we rely on potential theoretic arguments to obtain the result. However, one can recover a substantial part of the probabilistic method via the ground state transform. In this case we even get a probabilistic type representation of the harmonic part and an alternative formula for the potential part in the Riesz decomposition, see Theorem~\ref{thm:repHarPot}.

For the second decomposition, the Riesz-Mokobodzki-Sibony decomposition, we are not aware of a discrete analogue. Although it is certainly well known in the context of random walks, again our  proof for Schrödinger operators relies on potential theoretic arguments rather than probabilistic ones. We expect this to be useful in the study of limits of superharmonic functions at the Martin boundary.

As an application we present a Brelot type theorem. In the continuum this theorem gives an equality for the charge of a superharmonic function in terms of the infimum of the quotient of the function and the Green's function. However, in contrast to the continuum setting we only get one inequality which we show to be strict.

The paper is structured as follows. In the next section, Section~\ref{s:setup}. we introduce the setting and present the main results. In Section~\ref{s:toolbox}, we study the fundamental tools to prove the main theorems such as the Dirichlet problem and greatest harmonic minorants. In Section~\ref{s:proofs} we prove the main theorems and in Section~\ref{sec:rep} we give a probabilistic type representation of the Riesz decompostion. Finally, in Section~\ref{s:application} we show a Brelot type theorem.

\section{Setting the Scene and Main Results}\label{s:setup}
In this section we present the underlying notions of this work and state the main results.
\subsection{Graphs, Schrödinger Operators and Subcriticality}
Let $X$ be an infinite set equipped with the discrete topology. 
Let a symmetric function  $b\colon X\times X \to [0,\infty)$ with zero diagonal be given such that $ b $ is locally summable, i.e., the vertex degree $ \deg $ satisfies $$ \deg(x)=\sum_{y\in X}b(x,y)<\infty$$  for all $ x\in X $. We refer to $ b $ as a \emph{graph} over $X$ and the elements of $X$ are called \emph{vertices}. A subset $W\sse X$ is called \emph{connected} with respect to the graph $b$, if for every vertices $x,y\in W$ there is a path ${x_0,\ldots ,x_n \in W}$, such that $x=x_0$, $y=x_n$ and $b(x_{i-1},x_i)>0$ for all $i\in\set{1,\ldots, n}$. Throughout this paper we will always assume that 
\begin{center}$X$ is connected with respect to the graph  $b$.\end{center}

The space of real valued functions on $W\subseteq X$ is denoted by $\mathcal{C}(W)$ and the space of functions with compact support in $W$ is denoted by $ \mathcal{C}_c(W)$. We consider $\mathcal{C}(W)$ to be a subspace of $\mathcal{C}(X)$ by extending the functions of $\mathcal{C}(W)$ by zero on $X\setminus W$. 

A strictly positive function $m\colon X\to (0,\infty)$ extends to a measure with full support via ${m(W)= \sum_{x\in W}m(x)}$ for $W\sse X$.


For $W\sse X$, let the space $ \mathcal{F}(W)=\mathcal{F}_{b}(W) $ be given by
\[ \mathcal{F}(W)= \{ f\in \mathcal{C}(X)\mid \sum_{y\in W} b(x,y)\abs{f(y)} < \infty  \mbox{ for all } x\in X  \}.\]
We set $\cF=\cF(X)$ and define  the \emph{(formal) Schrödinger operator} $ H=H_{b,c,m}$ on $\cF$ via
\begin{align*}
Hf(x)=\frac{1}{m(x)}\sum_{y\in X} b(x,y)\bigl( f(x)-f(y)\bigr)+\frac{c(x)}{m(x)}f(x),\qquad x\in X,
\end{align*}
where $ c \in \mathcal{C}(X)$ is a function. 
 A function $u\in \mathcal{F}$ is  \emph{harmonic, (superharmonic, subharmonic)} on $W\sse X$ if \[Hu=0 \quad(Hu\ge 0,\, Hu\leq 0)\text{ on }W .\]
The operator $ H_{b,c,m} $ is said to be \emph{non-negative}
 on $  \cC_{c} (X)$ if for all $ \phi\in\cC_{c}(X) $ we have
 \begin{align*}
 \sum_{x\in X}(H_{b,c,m}\phi)(x) \phi(x) m(x)\ge 0,\qquad x\in X.
 \end{align*}
 By the Allegretto-Piepenbrink theorem, \cite{HK11, KePiPo1}, this is equivalent to the existence of a positive superharmonic function.


\subsection{Green's Functions and Potentials}\label{s:Green}

Let a graph $ b $ and a function $ c $ be given such that $ H $ is non-negative. We denote the restriction of  the operator $ H $ to $ \mathcal{C}(K) $ for a finite set $ K $ by $ H^{K} $. It is not hard to see that $ H^{K} $ is invertible on $ \mathcal{C}(K) $ due non-negativity of $ H $ and   the connectedness of $ X $, confer \cite[Lemma~5.15]{KePiPo1}. Furthermore, due to domain monotonicity one has $ (H^{K})^{-1}\ph\leq (H^{L})^{-1}\ph  $ for $\ph\in \mathcal{C}_{c}( K) $ for $ K\subseteq L $.

Let an increasing exhaustion  $(K_n)$ of $X$ with finite sets be given. We define the function $G=G_{b,c,m}\colon X\times X \to [0,\infty]$  via
\[G(x,y)=\lim_{n\to\infty}\bigl(H^{K_n}\bigr)^{-1}1_{y}(x), \]
for  $x,y\in X$, where $1_y$ is the characteristic function at $y\in X$, confer \cite[Theorem~5.16]{KePiPo1}.

Indeed, $ G $ is independent of the choice of $ (K_{n}) $.
In general $ G $ can take the value $ \infty $ but if  $G(x,y)< \infty$ holds for some  $x,y\in X$, then it  holds for all $ x,y\in X $, see \cite[Theorem~5.12]{KePiPo1}. Moreover, if the Green's function is finite for some measure $ m $ then it can be checked that it is finite for all measures, confer \cite{Schmidt}.

We call $ G=G_{b,c,m}  $ the \emph{Green's function} of $H=H_{b,c,m} $ in $X$. We call the operator $ H $   \emph{subcritical} if for some (all) $ x,y\in X $ we have
\begin{align*}
G(x,y)<\infty.
\end{align*}

\begin{remark}\label{r:subcritical}
	In  \cite[Theorem~5.3]{KePiPo1} it is shown that subcriticality is equivalent to the validity of a Hardy inequality, i.e., the existence of a function $ w\ge0 $ on $ X $	such that
	\begin{align*}
	\sum_{x\in X}(H\phi)(x)\phi(x)m(x)\ge \sum_{x\in X}w(x) \phi^2(x),\qquad \phi\in \cC_{c}(X).
	\end{align*}
	 Moreover, subcriticality is equivalent to the existence of  a least two linearly independent positive superharmonic functions and thence it implies that the corresponding Schrödinger operator is also non-negative.
	
	In probability theory one considers graphs with $c\geq 0$ and the measure $ m=\deg+c $. In this context a graph with a subcritical Schrödinger form is called \emph{transient}, see e.g. \cite{Fitzsimmons,Fuku,KePiPo1,Soardi,WoessRandom,WoessMarkov}, for an elaborate study on transience.
\end{remark}

\begin{remark}
	There are equivalent formulations of the Green's function via resolvents  or semigroups of a  self-adjoint realization $H_m$ on $\ell^2(X,m)$ of the Schrödinger operator $H$, i.e., 
	\[G(x,y)=\lim_{\alpha\downarrow 0}(H_m+\alpha)^{-1} 1_y(x)= \int_0^\infty \ee^{-tH_m} 1_y(x)\,\dd t,\qquad x,y\in X.\]
	Here, $\ell^2(X,m)$ denotes the space of square $ m $-summable functions.	For details see \cite{KePiPo1}.
\end{remark}

The function $ G $  is strictly positive, symmetric with respect to $ m $, superharmonic and if $H$ is subcritical then for all $y\in X$
\[HG(y,\cdot)=HG(\cdot,y)=1_y.\]
Moreover, for fixed $ y\in X $, the function $G(\cdot,y)$ is the smallest  function $ u\ge0 $ in $ \mathcal{F} $ such that $Hu\geq 1_y$, see \cite[Theorem~5.16]{KePiPo1}.

We denote the space of \emph{$G$-integrable} functions on $X$ by $ \mathcal{G}=\mathcal{G}_{b,c,m} $
\[\cG=\{f\in \cC(X)\mid \mbox{$ \sum_{y\in X} G(x,y)\abs{f(y)}<\infty $ for all $x\in X$}\}.\] 
Clearly, $ \cG $ is non-empty if and only if $H$ is subcritical in which case it obviously includes $ \cC_{c}(X) $. For $ f\in \mathcal{G} $ and $ x\in X $, we denote 
\begin{align*}
Gf(x)=\sum_{y\in X} G(x,y){f(y)}.
\end{align*}
Decomposing $ f \in \cG$ into positive and negative parts and approximating these parts monotonously  via compactly supported functions we have $ Gf\in \mathcal{F} $ and 
$$ HGf=f $$ by monotone convergence.

\begin{definition}
Let $H=H_{b,c,m}$ be subcritical. A function $p\in \mathcal{C}(X)$ is called a   \emph{potential} if  there is $f\in \cG$ such that \[p=Gf .\] The function $f$ is then called a \emph{charge of }$p$ and $ p  $ the \emph{potential of} $ f $.
\end{definition}

\subsection{Main Results}

In this subsection we present the main results of the paper which are two decompositions for superharmonic functions. Both are known as Riesz decompositions in the literature.

The first decomposition allows us to decompose a superharmonic function  into a harmonic and a potential part (provided the superharmonic function is bounded from below by a subharmonic function). 
Recall that a function is called a minorant of another function if it is smaller or equal everywhere.

The second decomposition, which is also known as the Mokobodzki-Sibony decomposition, states that if a the sum of two positive superharmonic functions has a positive superharmonic minorant, then this minorant can be decomposed into two positive superharmonic minorants of the original superharmonic functions.

Recall that a function $g\in \mathcal{C}(X)$ is called \emph{minorant} of $f\in \mathcal{C}(X)$ if $g\leq f$. Moreover, $h\in \mathcal{C}(X)$ is called  \emph{greatest harmonic minorant} of $f\in \mathcal{C}(X)$  if $h$ is a harmonic minorant of $f$ and for all other harmonic minorants $g$ of $f$ we have $g \leq h$. Clearly, the greatest harmonic minorant is unique in case it exists and we write 
\[h=\ghm_f.\]

Next, we present our first main result, the Riesz decomposition.

\begin{theorem}[Riesz Decomposition]\label{thm:rieszdecompGHM}
Let $H=H_{b,c,m}$ be subcritical and let $s$ be a superharmonic function with a subharmonic minorant. Then there exists a unique decomposition \[s = s_p+s_h\] such that $s_p$ is a non-negative potential with charge $Hs\in\cG$  and $s_h=\ghm_{s}$, in particular, $ s $ has a greatest harmonic minorant. 
Moreover, if $s\ge 0 $ then $s_{h}\geq 0$, and $s$ is the potential of a non-negative charge if and only if ${s}_{h}=0$.
\end{theorem}

\begin{remark}As stated in the theorem the  potential part is given by $s_p=GHs$. Furthermore, it becomes clear from the proof that the harmonic part arises as $s_h=\lim_{n\to\infty}s_n$, where $s_n$ is the solution of the Dirichlet problem with respect to $s$ on an increasing exhaustion $(K_n)$ of $X$ with finite sets, see Theorem~\ref{thm:propertiesGHM}.
\end{remark}

\begin{remark}
For a continuous analogue of Theorem~\ref{thm:rieszdecompGHM}, see \cite[Theorem~4.4.1]{AG} or  also \cite[Theorem~3.5.11]{HelmsNeu}. In these works one proves the decomposition by approximating the charge of the potential part by compactly supported functions and taking the monotone limit. We use a similar strategy to prove the decomposition in the discrete case. The continuous version of this theorem goes back to \cite[p. 350]{Riesz}. There is also a discrete version of this theorem in the context of transient random walks, i.e., $c\geq 0$, see e.g. \cite[Theorem~6.43]{WoessMarkov}, \cite[Theorem~24.2]{WoessRandom}, \cite[Theorem 1.38]{Soardi}. There, the superharmonic function is assumed to be non-negative.
\end{remark}

\begin{remark} Let us comment on a related decomposition in the literature. 
	In \cite{Soardi} (see also \cite{GHKLW}) the so called Royden decomposition is proven for functions of finite energy.  It states that such a function can be decomposed into a function which can be approximated by compactly supported functions and a harmonic functions. The proof relies on the Hilbert space structure of the functions of finite energy which we do not have at our disposal in our situation. 	
\end{remark}

Our second main result is the so-called Riesz-Mokobodzki-Sibony Decomposition.
\begin{theorem}[Riesz-Mokobodzki-Sibony Decomposition]\label{thm:MS}
Let $H_{b,c,m}$ be subcritical and let $s, s_1, s_2\ge0$ be superharmonic functions such that $s\leq s_1+s_2$. Then there exist unique  superharmonic functions  $0\leq r_1\leq s_{1}$  and  $0\leq  r_2\le s_{2}$ such that  \[  s=r_1+r_2.\]
\end{theorem}
For the continuous analogue, see \cite[Theorem 4.6.9]{HelmsNeu}.
\begin{remark}
 The name of the Riesz-Mokobodzki-Sibony decomposition seems to have changed over the years.    In \cite{Boboc} and \cite{Hansen} the continuous version of this theorem is called Riesz decomposition. In \cite{Hansen} this decomposition is also one of the four assumptions of a so-called balayage space. In \cite{HelmsNeu} the name Mokobodzki-Sibony theorem is used. The first proof of such a theorem in the continuum seems to go back to  \cite{MS68}.
\end{remark}

\begin{remark}
These Riesz decompositions are the fundamental tools to develop a Choquet-Martin boundary theory which generalises the existing theory in the probabilistic case, see \cite{WoessRandom}, to all graphs with corresponding subcritical Schrödinger operator, see \cite[Chapter 5]{Masterarbeit}. For instance, Theorem~\ref{thm:rieszdecompGHM} is the crucial step to get the so-called discrete Poisson-Martin integral representation. To gain certain boundary limits, Theorem~\ref{thm:MS} is of peculiar interest, see \cite[Section~4.4 and Chapter~6]{Masterarbeit}.
\end{remark}

\section{Toolbox}\label{s:toolbox}
The proofs of the theorems above are inspired by classical potential theoretic arguments in the continuum case as they can be found \cite{AG} and \cite{HelmsNeu}. To this end, solutions of a  Dirichlet problems along an exhaustion play a crucial role.

For the remainder of the section let $ b $ be a connected graph and let $ c$ be a function such that the operator $ H =H_{b,c,m}$ is {non-negative}.

\subsection{Dirichlet Problems on Finite Subgraphs}
The \emph{Dirichlet problem} on $W\sse X$ with respect to $f\in\mathcal{C}(X)$ is the problem of finding a function $u\in\mathcal{C}(X)$ such that 
\begin{align*}
\begin{cases} Hu= 0 &\text{ on } W, \\ \phantom{H}u=f &\text{ on } X\setminus W.\end{cases}
\end{align*}
The function $ u\in\mathcal{C}(X)$ is then referred to as the \emph{solution of the Dirichlet problem on} $ W $ \emph{with respect to} $ f $.

It is well known that due to positivity of $ H $ and connectedness of the graph these Dirichlet problems always have a unique solution. For the convenience of the reader we provide a short argument.

\begin{lemma}[Existence of Unique Solutions to Dirichlet Problems]\label{lem:dirichletProblem}
Let $H_{b,c,m}$ be non-negative on $\mathcal{C}_c(X)$, let $K\sse X$ be finite and let $f\in\mathcal{F}$. Then there exists a unique solution $ u $ to the Dirichlet problem  on $K$ with respect to $f$. Moreover, if $ f\ge0 $, then $ u\ge0 $.
\end{lemma}
\begin{proof}
	By a direct calculation, one sees that $ u $ is a solution  to the Dirichlet problem if and only if $ u $ satisfies
	\begin{align*}
	H_{b,c+d,m}u=g\quad \mbox{	on $ K $,}
	\end{align*}
 where $ d(x)=\sum_{y\in X\setminus K}b(x,y) $ and $ g(x)=\frac{1}{m(x)}\sum_{y\in X\setminus K}b(x,y)f(y) $ for $ x\in K $ and $d=g=0  $ on $ X\setminus K $. Note that the sum in the definition of  $ g $ converges absolutely due to the assumption $ f\in \mathcal{F} $.  
 Since the restriction $ H_{b,c+d,m}^{K} $ is invertible on $ \cC_{c}(K) $, confer \cite[Lemma~5.15]{KePiPo1}, and the resolvent is positivity preserving, confer \cite[Corollary~3.5]{KePiPo1}, we obtain the result.
\end{proof}
An important tool for the following potential theory on graphs is the so-called minimum principle.

\begin{theorem}[Minimum Principle, Lemma 5.14 in \cite{KePiPo1}] \label{thm:minprinc}
Let $H=H_{b,c,m}$ be non-negative on $\mathcal{C}_c(X)$. If $u\in \mathcal{F}$ satisfies $(H+\alpha)u\geq 0$ for $  \alpha \ge0 $ on a finite set $K\subseteq X$ and $u\geq 0$ on $X\setminus K$, then 
 $u=0$ or $u>0$ on $K$.
\end{theorem}

The minimum principle has the following immediate corollary.
\begin{corollary}\label{cor:dirichletProblem}
Let $H_{b,c,m}$ be non-negative and let $ s $  be a superharmonic function.
Then 
for any finite set $ K $ the unique solution $u$ to the Dirichlet problem  on $K$ with respect to  $s$ satisfies \[ u\leq s .\] 
\end{corollary}
\begin{proof}
Let $ u $	be the unique solution of the Dirichlet problem on a finite set $ K $ with respect to $ s $. On $K$, we have $H(s-u)\geq 0$ and on $X\setminus K$ we have $s-u=0$. By the minimum principle, Theorem~\ref{thm:minprinc}, we get that $s \geq u$ on $K$.
\end{proof}
\subsection{Existence and Properties of Greatest Harmonic Minorants}

In this section we study greatest harmonic minorants of superharmonic functions. Specifically, we show that these greatest harmonic minorants exist whenever there is a subharmonic minorant. Moreover, we prove that greatest harmonic minorants can be approximated by solutions of Dirichlet problems, are additive and monotone. Later we show that the greatest harmonic minorant of the Green's function is the zero function. 
These results are well known in classical potential theory in the continuum, see \cite[Section~3.3]{Helms}, but they seem to be new in the setting of Schrödinger operators on graphs.



\begin{theorem}[Existence and Properties of Greatest Harmonic Minorants]\label{thm:propertiesGHM}
Let $H=H_{b,c,m}$ be non-negative on $\mathcal{C}_c(X)$. Let $s$ be superharmonic with subharmonic minorant $ u $. Then $s$ has a greatest harmonic minorant $\ghm_s$ such that $ \ghm_{s}\ge u $. 

Moreover, we have 
\[ \ghm_s=\lim_{n\to\infty}s_n\]
as a pointwise limit, where $s_n$ is the solution of the Dirichlet problem with respect to $s$ on an increasing exhaustion $(K_n)$ of $ X $ with finite sets.

Furthermore, greatest harmonic minorants are additive, i.e., if $t$ is superharmonic with a subharmonic minorant, then $\ghm_{s+t}$ exists and 
\[ \ghm_{s+t}=\ghm_s+\ghm_t.\]

Moreover, greatest harmonic minorants are monotone, i.e., if $s\leq t$, then
\[\ghm_s \leq \ghm_t.\]
\end{theorem}
\begin{proof}
Let $(K_n)$ be an increasing exhaustion of $X$ with finite sets and let $s_n$ denote the solution of the Dirichlet problem on $K_n$ with respect to  $s$ which exists for every $n\in\mathbb{N}$ by Lemma~\ref{lem:dirichletProblem}.

Firstly, we show the existence of the greatest harmonic minorant: By Corollary~\ref{cor:dirichletProblem}, we get  $s \geq s_n$ on $K_n$ for every $n\in \mathbb{N}$. On $K_{n}$, we have $H(s_{n}-s_{n+1})=0$ and  on $X\setminus K_n$, we have $s_n-s_{n+1}=s-s_{n+1}\geq 0$. Hence, by the minimum principle, Theorem~\ref{thm:minprinc}, we get  $$ s_n\geq s_{n+1}\mbox{ on }X $$ for every $n\in\mathbb{N}$ which means that we have an decreasing sequence of harmonic functions.
Thus, there exists a pointwise limit $$ s_\infty=\lim_{n\to\infty}s_n $$ which might take the value $ -\infty $ on vertices. However, this is not the case: By assumption $s$ has a subharmonic minorant which we denote by $ u $. Then,  we have $H(s_n-u)\geq 0$ on  $K_n$ and $s_n-u=s-u\geq 0$  on $X\setminus K_n$, $ n\in\mathbb{N} $. By the minimum principle, Theorem~\ref{thm:minprinc}, we obtain $s_n\geq u$ on $X$ for all $n\in\mathbb{N}$. So, $$ s\geq s_\infty\geq u $$  which shows finiteness of $ s_{\infty} $.  By monotone convergence, we find that the function $s_\infty$  is in  $\mathcal{F}$ and  \[0= \lim_{k\to\infty} Hs_k= Hs_\infty.\]
Thus, $ s_\infty $ is a harmonic minorant.

Finally, let $ v $ be another harmonic minorant. Then again, on $ K_{n} $, we have $ H(s_{n}-v)=0 $ and $ s_{n}-v=s-v\ge0 $ on $ X\setminus K_{n} $. Hence, the minimum principle,  Theorem~\ref{thm:minprinc}, yields $ s_{n}\ge v $ and, therefore, $ s_{\infty}\ge v $. Thus,
\begin{align*}
s_{\infty}=\ghm_{s}.
\end{align*}
\smallskip

Secondly, we show the additivity property: Consider the solutions of the Dirichlet problem $s_n$ and $t_n$ on $ K_{n} $ with respect to $ s $ and $ t $,  $n\in \mathbb{N}$. Since $s_n+t_n$ solves  the Dirichlet problem 
\begin{align*}
\begin{cases} Hw= 0 &\text{ on } K_n, \\ \phantom{H}w=s+t &\text{ on } X\setminus K_n, \end{cases}
\end{align*}
we obtain by  the above
\begin{align*}
\ghm_{s+t}=(s+t)_{\infty}=\lim_{n\to\infty}(s_{n}+t_{n})=\lim_{n\to\infty}s_{n}+\lim_{n\to\infty}t_{n}=s_{\infty}+t_{\infty}=\ghm_{s}+\ghm_{t}.
\end{align*}
Thirdly, we show  monotonicity: If $ s\le t $, then $ \ghm_{s} $ is a harmonic minorant of $ t $. Thus, $ \ghm_{s}\leq \ghm_{t} $ as $ \ghm_{t} $ is the greatest harmonic minorant of $ t $.
\end{proof}

\subsection{Greatest harmonic minorants of potentials}

Next, we relate greatest harmonic minorants and potentials. Recall that if $ H_{b,c,m} $ is subcritical, then it is positive by the Allegretto-Piepenbrink theorem, \cite[Theorem~4.2]{KePiPo1} (see Remark~\ref{r:subcritical} as well). Moreover,   the Green's function exists in this case.
\begin{theorem}\label{thm:GHM}
Let $H_{b,c,m}$ be  subcritical. Then, for  $ f\in \mathcal{G} $,
\begin{itemize}
\item[(a)] 
 $ \ghm_{Gf}$ exists, 
\item[(b)]   $ \ghm_{Gf}\leq 0$,
\item[(c)]   $\ghm_{Gf}=0$  if $ f\ge0 $. 
\end{itemize}
\end{theorem}

For the proof of this theorem we need a local Harnack inequality which is well known in the context of graphs see \cite[Theorem 4.5]{KePiPo1} as well as \cite{HK11} and references therein.

\begin{proposition}[Harnack Inequality, Theorem 4.5 in \cite{KePiPo1}]\label{thm:harnackineq} Let $ H_{b,c,m} $ be non-negative, let $K\sse X$ be a connected and finite set and let $f\in\mathcal{C}(X)$. Then there exits a constant $C=C(f)>0$ such that the such that for any $u\in\mathcal{F}$, $u\geq 0,$ such that $(H-f)u\geq 0$ on $K$ we have
	\[ \max_K u \leq C \min_K u. \]
	The constant $C(f)$ can be chosen  monotonously in the sense that if $f\leq g$ then $ C(f) \geq C(g).$
\end{proposition}

We recall from the discussion
in Section~\ref{thm:GHM} that $ G(\cdot,y) $ is the smallest $ v\ge0 $ such that $ Hv =1_y$, $ y\in X $.

\begin{proof} [Proof of Theorem~\ref{thm:GHM}]	
For a function $ f ,$ we let $ f=f_{+}-f_{-} $, where $ f_{\pm}= \max\{\pm f,0\} $.
	
Ad~(a):
Let $ f\in \mathcal{G} $ be given. Then, the function $ -Gf_{-}$ is a subharmonic minorant to $  G f $ since $ G f=Gf_{+}-Gf_{-}\ge -Gf_{-} $ and $ H(-Gf_{-}) =-f_{-} \leq 0$. Hence, the greatest harmonic minorant of $ Gf $ exists by Theorem~\ref{thm:propertiesGHM}.

Ad~(b) for $ f\in \mathcal{C}_{c}(X) $: We show the statement for $ f=1_{y} $, $ y\in X $,  first and prove the statement for $ f\in \mathcal{C}_{c} (X)$ afterwards. The statement for general $ f\in \mathcal{G} $ is then proven after we have shown (c).

 Let $y\in X$. We note that $0$ is a harmonic minorant of the positive superharmonic Green's function $G(\cdot,y)$ and let $u$ be an arbitrary harmonic minorant of $G(\cdot,y)$. Then $H(G(\cdot,y)-u)=1_y$. Since $ G(\cdot,y) $ is the smallest solution $ v\ge0 $ to $ Hv=1_y $, it follows that $G(\cdot,y)-u\geq G(\cdot,y)$. Hence, $u\leq 0$ and $\ghm_{G(\cdot,y)}=0$.

Now, let $K$ be a connected and finite set such that $\supp(f)\sse K$.
Since the functions $G(x,\cdot)  $ and $ G(\cdot,x) $ are positive and superharmonic for all $ x,y\in K $ and $ K  $ is finite, we obtain by the Harnack inequality, Proposition~\ref{thm:harnackineq},  the existence of a constant $C> 0$ such that $$ G(x,z)\leq C\cdot G(x,y) $$ for all $x,y,z\in K$. 
This implies for $x,y\in K$
\[Gf(x)\leq \max_{z\in K}|f(z) |\sum_{z\in K}G(x,z)\leq C \cdot \# K \cdot \max_{z\in K}|f(z)| G(x,y).\] 
Let $ u $ be the harmonic minorant of $ Gf $. Then, $ u $ is a harmonic minorant of $ G(\cdot,y) $ and, therefore, $ u\leq \ghm_{G(\cdot,y)} $. But above  we have shown $\ghm_{G(\cdot,y)}=0$.

Ad~(c): Since $ f\ge0 $, we have $ Gf\ge0 $ and $ Gf $ is superharmonic. Therefore, $ \ghm_{Gf} $ exists as $ 0 $ is a harmonic minorant for $ Gf $.
Let $(K_n)$ be an increasing exhaustion of $X$ with finite sets and define $f_n=1_{K_n}f$, $n\in\mathbb{N}$. By (b), for $  \mathcal{C}_{c}(X) $, we know that $\ghm_{Gf_n}\leq 0$ and since $ 0 $ is a harmonic minorant we have $ \ghm_{Gf_{n}}=0 $.  Since $Gf, Gf_n, G(f-f_n)$ are non-negative superharmonic functions and $ Gf= G(f-f_n)+Gf_{n}$, we can use the the additivity of greatest harmonic minorants 
to get
\begin{align*}
0\leq \ghm_{Gf}=\ghm_{G(f-f_{n})}+\ghm_{Gf_{n}}=\ghm_{G(f-f_{n})}\leq G(f-f_{n})
\end{align*}
By monotone convergence, $G(f-f_n)\to 0$ as $n\to \infty$. So, we conclude that $\ghm_{Gf}=0$.  

Ad~(b) for general $ f\in \mathcal{G} $: Let $ f\in \mathcal{G} $. By (a) the greatest harmonic minorants of $ Gf_{+} $ and $ G(-f_{-})=-Gf_{-}$ exist and by (c) we have  $ \ghm_{Gf_{+}}=0 $. By additivity of the greatest harmonic minorants, Theorem~\ref{thm:propertiesGHM},  we have
\begin{align*}
\ghm_{Gf}=\ghm_{Gf_{+}}+\ghm_{-Gf_{-}}=\ghm_{-Gf_{-}}\leq {-Gf_{-}}\leq 0.
\end{align*}
This finishes the proof.
\end{proof}
Theorem~\ref{thm:GHM} states that any potential has a greatest harmonic minorant. Recall that the Riesz decomposition theorem, Theorem~\ref{thm:rieszdecompGHM}, says that any superharmonic function with subharmonic minorant can be decomposed into a potential part with non-negative charge and a harmonic part. Hence, the greatest harmonic minorant of this potential part is the zero function.

\section{Proofs of the Main Results}\label{s:proofs}
\subsection{Proof of the Riesz Decomposition}
%
We next prove one of the main theorems, the Riesz decomposition, Theorem~\ref{thm:rieszdecompGHM}.

\begin{proof}[Proof of Theorem~\ref{thm:rieszdecompGHM}]
Firstly, we assume that $Hs\in \cG$ (and we show below that this is always the case). Then $GHs$ is a non-negative superharmonic function and we can apply $H$ to it. We show that  
\[u=s-GHs\]
is the greatest harmonic minorant of $ s $.
We have $Hu=H(s-GHs)=0$, so $u$ is a harmonic  function. Moreover, by Theorem~\ref{thm:GHM}, we have $\ghm_{GHs}=0$. Thus, using additivity of the greatest harmonic minorants  we get 
\[ u=\ghm_u=\ghm_u+\ghm_{GHs}=\ghm_{u+GHs}=\ghm_ s.\]
Hence, $ s_{h}= (s-GHs)$ and $ s_{p} =GHs$,  which shows the existence of the decomposition. 

As for the uniqueness, let $ s=Gf+h $ be another decomposition with $ f\in \mathcal{G} $ and $ h $ harmonic. Then, $ \ghm_{s}-h=G(f-Hs) $ is harmonic and therefore,
\begin{align*} 
0=H(\ghm_{s}-h)=HG(f-Hs)=f-Hs. 
\end{align*} 
We infer $ f=Hs $ which readily implies $  h=\ghm_{s}$.

Furthermore,  $s\geq 0$ implies $\ghm_{s}\geq 0$ since $ 0 $ is a harmonic minorant in this case. Moreover, by Theorem~\ref{thm:GHM} we get that $s$ is the potential of a non-negative charge if and only if $\ghm_{s}=0$ and the theorem for $Hs\in\cG$ is proven.

To finish the proof we show that $Hs\in\cG$. The  idea is to find an upper bound for $GHs$. Let $(K_n)$ be an increasing exhaustion of $X$ with finite sets. Since we assumed that $ s $ has a subharmonic minorant,  $ \ghm_{s} $ exists by  Theorem~\ref{thm:propertiesGHM}. Consider the function
\begin{align*}
v_{n}=s-\ghm_{s}-G(1_{K_{n}}Hs)
\end{align*}
for $ n\in\mathbb{N} $. Since
\begin{align*}
Hv_{n}=Hs-1_{K_{n}}Hs\ge0,
\end{align*}
and $ v_{n} $ has a subharmonic minorant with $-G(1_{K_{n}}Hs  )$, the greatest harmonic minorant $ \ghm_{v_{n}} $ exists  by  Theorem~\ref{thm:propertiesGHM} and
\begin{align*}
-\ghm_{v_{n}}\leq G(1_{K_{n}}Hs  ).
\end{align*}
Since the greatest harmonic minorant of a potential is non-positive, Theorem~\ref{thm:GHM}, we have $- \ghm_{v_{n}}\leq0 $ and, therefore,
$ \ghm_{v_{n}}\ge0 $.
Thus, we obtain
\begin{align*}
G(1_{K_{n}}Hs)=s-\ghm_{s}-v_{n}\leq s-\ghm_{s}-\ghm_{v_{n}}\leq s-\ghm_{s}.
\end{align*}
Monotone convergence yields $ GHs\leq s-\ghm_{s} $ and, therefore, $ Hs\in \mathcal{G} $. This finishes the proof.
\end{proof}

\subsection{Proof of the Riesz-Mokobodzki-Sibony Decomposition}
Next, we prove the Riesz-Mokobodzki-Sibony decomposition. It states that if positive superharmonic functions $ s,s_{1},s_{2} $ satisfy $s\leq s_{1}+s_{2}  $, then $ s=r_{1}+r_{2} $ with  superharmonic functions $0\leq r_{1}\leq s_{1} $ and $0\leq r_{2}\leq s_{2} $. The proof is inspired by the one of  \cite[Theorem~4.6.9]{HelmsNeu} in the continuum setting of $ \mathbb{R}^{n} $.

We start the proof with two simple observations which will be used in the proof. The first lemma, Lemma~\ref{lem:deg+q<=0}, uses the existence of a strictly positive superharmonic function which follows directly from the Allegretto-Piepenbrink theorem, see \cite[Theorem 4.2]{KePiPo1}. 
\begin{lemma}\label{lem:deg+q<=0}
	Let $H=H_{b,c,m}$ be non-negative on $\mathcal{C}_c(X)$. Then $\deg + c> 0$. 
\end{lemma}
\begin{proof}
	By the Allegretto-Piepenbrink theorem  there exists a strictly positive superharmonic function $s$. Then $Hs\geq 0$ implies for all $x\in X$ that
	\[   \bigl( \deg(x)+c(x)\bigr)s(x)\ge \sum_{y\in X}b(x,y)s(y)> 0\]
	by the strict positivity of $ s $.
\end{proof}
The next lemma is well known, but we include the short proof for the convenience of the reader.
\begin{lemma}\label{lem:infsuberharmonic}  Let $H=H_{b,c,m}$ be non-negative on $\mathcal{C}_c(X)$. Let $T$ be a subset of the set of non-negative superharmonic functions. Then, the function   \[ {r(x)=\inf_{t\in T} t(x)},\qquad {x\in X}, \] is a non-negative superharmonic function.
\end{lemma}
\begin{proof}
	Let $t\in T$ be fixed and $r$ as stated above. It is obvious that $r\geq 0$ and $r\in \mathcal{F}$. So, it remains to show that $Hr\geq 0$. By Lemma~\ref{lem:deg+q<=0} we have $\deg +c> 0$ and  $Ht\ge0$ implies 
	\[ t(x)\geq \frac{1}{\deg(x)+c(x)}\sum_{y\in X}b(x,y)t(y)\geq \frac{1}{\deg(x)+c(x)}\sum_{y\in X}b(x,y)r(y) \]
	for all $x\in X$. Taking the infimum over all $t\in T$ on the right-hand side yields  $Hr\geq 0$.
\end{proof}

\begin{proof}[Proof of Theorem~\ref{thm:MS}] We first show existence and then uniqueness.\medskip
	
\emph{Existence:}  We define 
 \[\cS_1=\set{v\geq 0\colon v\text{ is superharmonic and } s-s_2\leq v} \quad \text{ and } \quad r_1=\inf\cS_1,\]
 as well as
 \[\cS_2=\set{v\geq 0\colon v\text{ is superharmonic and }s-r_1\leq v } \quad \text{ and } \quad r_2=\inf\cS_2,\]
 where the infimums are taken pointwise. 	By Lemma~\ref{lem:infsuberharmonic}, the functions $ r_{1} $ and $ r_{2} $ are non-negative superharmonic and, thus, $ r_{1}\in \mathcal{S}_{1} $ and $ r_{2}\in \mathcal{S}_{2} $.\medskip
	
We show $ s=r_{1}+r_{2} $ by proving several claims.	\medskip
	
\emph{Claim 1. We have $s\leq r_{1}+r_{2}$.}\\\emph{Proof of the claim.}
This follows directly from  $ r_{2}\in \mathcal{S}_{2} $.\medskip

It remains to show that $s\geq r_1+r_2$. This is done in several steps.\medskip

\emph{Claim 2. If $s-r_2$ is superharmonic, then $ s\ge r_{1}+r_{2} $.}\\\emph{Proof of the claim.} Since $s\in \mathcal{S}_{2} $, we obtain $$  s-r_{2}\geq 0 . $$ On the other hand since  $ r_{1}\in \mathcal{S}_{1} $ we have $ s-s_{2}\leq r_{1} $ and, therefore, $ s- r_{1}\leq s_{2} $ which yields $ s_{2}\in \mathcal{S}_{2}. $  We obtain $ r_{2}\le s_{2} $ and therefore,
\begin{align*}
s-s_{2}\leq s-r_{2}
\end{align*}
So, if $ s-r_{2} $ is superharmonic, then $ s-r_{2}\in \mathcal{S}_{1} $. As a consequence, $ s-r_{2}\ge r_{1} $ which yields
\begin{align*}
s\ge r_{1}+r_{2}.
\end{align*}
This proves the claim.\medskip

To prove that $ s-r_{2} $ is superharmonic  we need the following notation. Recall that  $\deg+ q> 0$  by Lemma~\ref{lem:deg+q<=0}.
For $ f\in \mathcal{F} $ and $ x\in X $, the function
\begin{align*}
f_x(z)=\begin{cases}\frac{1}{\deg(x)+c(x)}\sum_{y\in X}b(x,y)f(y) &: z=x,\\ \quad f(z)&: z\neq x. \end{cases}
\end{align*}
solves Dirichlet problem on $ \{x\} $ with respect to $ f $ which can be seen by the following direct calculation
\begin{align*}
Hf_{x}(x)=\bigl(\deg(x)+c(x)\bigr)f_{x}(x)- \sum_{x\in X}b(x,y)f_{x}(y)= 0,\qquad x\in X.
\end{align*}
For what follows, we write $r_{1,x}=(r_1)_x$ and $ r_{2,x}=(r_{2})_{x} $.

\textit{Claim 3. We have  $r_{2}-r_{2,x}\leq s-s_{x}$.} \\\emph{Proof of the claim.} Consider the following auxiliary function
\[  w= \min\set{r_2, r_{2,x}+s-s_{x}}.\]
Note that $ w=r_{2} $ on $ X\setminus\{x\} $ as $ r_{2}=r_{2,x} $ and $ s=s_{x} $ outside of $ x $. Clearly, $ w\leq r_{2} $. So if we show that $ w\in \mathcal{S}_{2} $, then $ w=r_{2} $ and the claim follows.
First of all, since the minimum of superharmonic functions is superharmonic, Lemma~\ref{lem:infsuberharmonic}, we have $ Hw\ge0 $. Furthermore, $ w\ge0 $ since $r_{2}\ge r_{2,x}\ge0 $
and $ s\ge s_{x} $ by Lemma~\ref{lem:dirichletProblem} and  Corollary~\ref{cor:dirichletProblem}. We are left to show that $ s-r_{1}\leq w $. On the one hand, $ s-r_{1}\leq r_{2} $ by Claim 1. Moreover, this inequality $s\leq r_1+r_2$, implies by Lemma~\ref{lem:dirichletProblem}
\[s_x\leq \bigl(r_1+r_2\bigr)_{x}=r_{1,x}+r_{2,x}.\]
Now, combining this inequality with $r_1\geq r_{1,x}$, Corollary~\ref{cor:dirichletProblem}, we get
\[ r_1+ r_{2,x}+s-s_{x} \geq r_{1,x}+ r_{2,x}+s-s_{x}\geq s.\]
Thus, we have $s-r_{1}\leq w$ which finishes the proof of the claim.\medskip

\textit{Claim 4: The function $s-r_2$ is superharmonic.} \\\emph{Proof of the claim.}
By Claim 3, $ r_{2}-r_{2,x}-s_{x}\leq s $ for all $ x\in X $ and, therefore,
\[s(x)-r_2(x)\geq s_{x}(x)-r_{2,x}(x)= (s-r_2)_{x}(x)= \frac{1}{\deg(x)+c(x)}\sum_{y\in X}b(x,y)(s-r_2)(y).\]
This shows that $ s-r_{2} $ is superharmonic in $ x. $  Applying this argument to every $x\in X$ we get that $s-r_2$ is superharmonic on $X$.
\medskip

In summary Claim 1 shows $ s\leq r_{1}+r_{2} $ and Claim~4 combined with Claim~2 yields $ s\ge r_{1}+r_{2} $. Hence, $ s=r_{1}+r_{2} $ and the proof of existence is finished.\medskip

We finally turn to uniqueness.\medskip

\emph{Uniqueness:} 
	Assume that there are non-negative superharmonic functions $t_1$ and $t_2$ for  which $t_1\leq s_1$, $t_2\leq s_2$ and $s=t_1+t_2$. Then $s\leq t_1 + s_2$ and $s\leq s_1+ t_2$. Hence, $ t_{1}\in \mathcal{S}_{1} $ and $ t_{2}\in \mathcal{S}_{2} $ which readily gives 	 $r_1\leq t_1$ and $r_2\leq t_2$. Since $r_1+r_2=t_1+t_2$, we get $r_1=t_1$ and $r_2=t_2$.
\end{proof}

\section{Representations of the Harmonic and the Potential Part}\label{sec:rep}
In this section we present a representation of the harmonic and the potential part of the Riesz decompostion. This representation is inspired by the corresponding result in the context of random walks, confer \cite{WoessRandom,WoessMarkov}. The validity of such a result is in this sense surprising as the semigroups of Schrödinger operators do not allow for a probabilistic interpretation in the case of non-positive $ c $. However, the main idea is to use a ground state transform. This way we get a decomposition to transfer the corresponding decomposition of the random walk context into the context of Schrödinger operators.

In this subsection let $ b $  and $ c $ be such that $H =H_{b,c,m} $ is non-negative. Hence,   $\deg + c > 0$ by Lemma~\ref{lem:deg+q<=0}. We define the function $p\colon X\times X \to (0,\infty)$ via 
\[p(x,y)=\frac{b(x,y)}{\deg(x)+c(x)} \] 
and the operator $P=P_{b,c}\colon \mathcal{F}\to \mathcal{C}(X)$ is defined by \[Pf(x)=\sum_{y\in X} p(x,y)f(y) \] for all $x\in X$.
In the setting of random walks, i.e., $c\geq 0$, the function $p$ is called the \emph{transition matrix} and the operator $P$ is called the \emph{transition operator}, \cite{WoessMarkov}.  

We consider the multiplication operator $D=D_{b,c,m}\colon \mathcal{C}(X) \to \mathcal{C}(X)$ defined via 
\[Df(x)=\frac{\deg(x)+c(x)}{m(x)} f(x)\]  which is invertible since  $\deg + c > 0$ by Lemma~\ref{lem:deg+q<=0}. We denote the inverse by $ D^{-1} $.
Note that $P, D, D^{-1}$ are positivity preserving and we have on $\mathcal{F}$
\[ H=D(I-P).\]
Letting $ n=\deg+c $, we readily see that $ I-P $ is a Schrödinger operator such that $ (I-P)=H_{b,c,n} $.
Obviously, for $ s\in \mathcal{F} $, we have $ Hs=0 $ (respectively $ Hs\ge0 $, $ Hs\leq0 $) if and only if $ (I-P )s\ge0 $ (respectively $ (I-P )s\ge0 $, $ (I-P )s \leq0 $).

In the Riesz decomposition theorem, Theorem~\ref{thm:rieszdecompGHM}, we have shown that a superharmonic function $ s $ with subharmonic minorant decomposes uniquely as
\begin{align*}
s=GHs+\ghm_{s}
\end{align*}
with  potential $ GHs $  and  harmonic part $\ghm_{s}$. Next, we give the main result of this section, an alternative representation of $ GHs$ and $ \ghm_{s}$.
Furthermore, fixing the measure $ m $, we denote the operator of multiplication by the function $ g m $ by $ M_{g} $. Specifically, we will use the operator
\begin{align*}
M_{s^{2}}f =  s^2m f.
\end{align*}

\begin{theorem}[Representation of the Harmonic and the Potential Part]\label{thm:repHarPot}
	Let $H_{b,c,m}$ be  subcritical. If $s>0$ is superharmonic, then  
\begin{align*}
GHs &= M_{s^{2}}DG M_{s^{-2}} (I-P)s=(M_{s^{2}}D)G(M_{s^{2}}D)^{-1}Hs \end{align*} 
and
\begin{align*}
\ghm_{s}&= \lim_{n\to\infty}P^ns.
\end{align*}	
	In particular, the  limit in the second equality exists. 
\end{theorem}
Note that if $H_{b,c,m}$ is subcritical there always exists $s>0$ superharmonic by the Allegretto-Piepenbrink theorem, see \cite[Theorem 4.2]{KePiPo1}.

The basic idea is to prove a corresponding version of the theorem in the case $ c\ge0 $. In this case a subcritical operator $ H_{b,c,m} $ is called transient. Furthermore, according to  \cite [Theorem~6]{Schmidt}, the Green's function of $ H=H_{b,c,1} $ satisfies for $ c\ge 0 $
\begin{align*}
G(x,y)=\frac{1}{\deg(x)+c(x)}\sum_{k=0}^\infty P^k1_y(x),
\end{align*}
for all $x,y\in X$.

The proof of the following theorem works along the lines of \cite[p. 169]{WoessMarkov}.

\begin{lemma}[Riesz Decomposition for  $c\geq 0$]\label{lem:rieszdecompq+deg>0}
	Let $b$ be a graph over $X$ and $c\geq 0$ such that $H_{b,c,1}$ is subcritical and let  $s$ be a superharmonic function with subharmonic minorant. Then, the monotone limit \[ s_{h}=\lim_{n\to\infty}P^ns \] exists pointwise and is harmonic, and \[  s_{p}=D^{-1}(s-s_{h})  \] is a potential with charge $  (I-P)s$.
	In particular, the Riesz decomposition of $ s $ with respect to the operator $  (I-P)$ is
	$ s=s_{h}+s_{p} $, and with respect to $H$ is $s=s_h+Ds_p$.
\end{lemma}
\begin{proof}
	Assume first that $s$ is non-negative. Note that $Hs\geq 0$ implies $ (I-P) s\ge0$ and, therefore, $s\geq Ps$. Hence, the limit
	$s_{h}= \lim_{n\to\infty}P^n s $ exists due to monotonicity and the fact that $ P $ is  positivity preserving. Moreover, we clearly have $ Ps_{h}=s_h $ due to monotone convergence which implies $ Hs_{h}=0 $. Moreover, we have by harmonicity of $ s_{h} $, i.e., $ (I-P)s_{h}=0 $, the representation of the Green's function above and $(\sum_{k=0}^{\infty}P^{k})(I-P)=I $, that
	\[  G(I-P)s= G(I-P)(s-s_{h})=D^{-1}\Bigl(\sum_{k=0}^{\infty}P^{k}\Bigr)(I-P)(s-s_{h})=D^{-1}(s-s_{h}).\]
	Hence, we obtain for superharmonic $ s\ge0 $, that
	\begin{align*}
	s=\lim_{n\to\infty}P^ns + DG(I-P)s,
	\end{align*}
which is the Riesz decomposition of $s$ with respect to the operator $I-P$. Hence, $\lim_{n\to\infty}P^ns$ is the greatest harmonic minorant of $s$ with respect to $I-P$ and therefore also with respect to $H$. But this implies that $DG(I-P)s=GHs$ is the corresponding potential of the decomposition with respect to $H$.	
	
	Assume now that $s$ is superharmonic with subharmonic minorant but not necessarily non-negative. Then by Theorem~\ref{thm:propertiesGHM} the greatest harmonic minorant $\ghm_s$ exists and $s-\ghm_s$ is a non-negative superharmonic function. Applying the first part of the proof yields 
	\begin{align*}
	s-\ghm_s&=\lim_{n\to\infty}P^n(s-\ghm_s)+DG(I-P)(s-\ghm_s)\\
	&=\lim_{n\to\infty}P^ns-\ghm_s+DG(I-P)s,
	\end{align*}
	where the second equality can be justified as follows: First of all $ \ghm_{s} $ is harmonic and therefore $ \lim_{n\to\infty}P^{n}\ghm_{s}=\ghm_{s}=P\ghm_{s} $. Secondly, since all other involved terms are finite  we conclude that  $\lim_{n\to\infty}P^ns$ exists and the equality follows. 
	
	Since the Riesz decomposition is unique by Theorem~\ref{thm:rieszdecompGHM}, we obtain the result.
\end{proof}

With the lemma above we can now deduce the statement of Theorem~\ref{thm:repHarPot} by the virtue of the so called ground state representation.
\begin{proof}[Proof of Theorem~\ref{thm:repHarPot}] Let $s>0$ be superharmonic with respect to $H$.
	We denote
	\begin{align*}
	H_{s}=H_{b,c,s^{-2}}=M_{s^{2}}H.
	\end{align*}
	Clearly, $ H $ and $ H_{s} $ share the same (super/sub-)harmonic functions and it is not hard to see that 
	\begin{align*}
	G_{s}=GM^{-1}_{s^{2}}
	\end{align*}
	is the Green operator of $ H_{s} $. Furthermore, we set
	\begin{align*}
	D_{s}=M_{s^{2}}D \qquad \mbox{and}\qquad P_{s}=P.
	\end{align*}
	Then, $ H_{s}=D_{s}(I-P_{s}) $.
	One readily sees (confer \cite[Section~4.2]{KePiPo1}) that the corresponding ground state representation $ H^{s}= H_{b^{s},c^{s},1}$  of $ H_{s} $ with respect to $ s $ associated with the graph
	\[ b^s(x,y)=b(x,y)s(y)s(x),\quad \mbox{and} \quad c^s(x)=s^{-1}H_{s}s(x)=M_{s}Hs(x),\qquad x,y \in X, \] 
	acts on $\mathcal{F}^s={s}^{-1}\mathcal{F}$ and satisfies \[ H^{s}f={s}^{-1}H_{s}(sf).  \]
	Specifically, $ c^{s}\ge0 $. Moreover,  a function $ u $ is (super/sub-)harmonic for $ H_{s} $ if and only if  $ (s^{-1}u) $ is  (super/sub-)harmonic for $ H^{s} $. Hence, we are in a position to apply the previous lemma, Lemma~\ref{lem:rieszdecompq+deg>0}, to $ H^{s}=H_{b^{s},c^{s},1} $. 
	
	But before doing so we need to consider the Green operator of $ H^{s} $ first which acts on $\cG^s=s^{-1}\cG$ via \[ G^sf=s^{-1}G_{s}(sf).\]
	Moreover, on $\mathcal{F}^s$, respectively $\mathcal{C}(X)$ we set \[P^sf=s^{-1}P_{s}(sf), \quad\text{ respectively }\quad  D^sf=s^{-1}D_{s}(sf). \] The operators $P^s$ and $D^s$ are the corresponding transition operator and degree matrix of $H^s$, i.e., $ H^{s}=D^{s}(I-P^{s}) $.

 By assumption the superharmonic function $ s $ has a subharmonic minorant $ u $ for  $ H $ and $ H_{s} $.
	Then the constant function $1=s^{-1}s$ is superharmonic  with subharmonic minorant $ s^{-1}u $ for $ H^{s} $. Therefore, we can apply Lemma~\ref{lem:rieszdecompq+deg>0} to $1=s^{-1}s$ to get
	\[1=  D^sG^s(I-P^s)1+\lim_{n\to\infty}(P^s)^n1 \]
	which is equivalent to
	\[ s= D_sG_s(I-P_s)s+\lim_{n\to\infty}P^ns.  \]
	Since $ 1_{h}=\lim_{n\to\infty}(P^s)^n1 $ is harmonic for $ H^{s} $, the function  \[s_{h}=s 1_{h}=s\lim_{n\to\infty}(P^s)^n1 =\lim_{n\to\infty}P^ns \] is harmonic for $ H_{s} $. Furthermore, according to Lemma~\ref{lem:rieszdecompq+deg>0} the function $  1_{h}$ is the greatest harmonic minorant of $1 $ with respect to the operator $ I-P^{s} $ and, therefore, with respect to the operator $ H^{s} $. Hence, it  follows that $ s_{h} =s 1_{h} $ is the greatest harmonic minorant of $ s $ with respect to $ H_{s} $ and $ H $: Indeed if $ Hu=0 $ with $s\ge u\ge s_{h} $, then $ H^{s}(s^{-1}u)=0 $ with $1\ge s^{-1}u\ge s^{-1}s_{h}=1_{h}$. Thus, $ s^{-1}u $ is the greatest harmonic minorant of $ H^{s} $ and therefore, $ u=s 1_h $ and $s_{h}= s1_{h} $ is the greatest harmonic minorant of $ H_{s} $ (and H).
	We conclude
	\begin{align*}
	\ghm_{s}=s_{h}=\lim_{n\to\infty}P^ns.
	\end{align*}
	By the Riesz decomposition, Theorem~\ref{thm:rieszdecompGHM}, we infer that the potential part $ s_{p}=G_{s}H_{s}s $ of $ s $ with respect to $ H_{s} $ equals \[  s_{p}=D_sG_{s}(I-P_{s})s=D_{s}G_sD_{s}^{-1}H_{s}s  .\]
	Now plugging in the equalities $ 	H_{s}=H_{b,c,s^{-2}}=M_{s^{2}}H $, $ 	G_{s}=GM^{-1}_{s^{2}}
	$, $	D_{s}=M_{s^{2}}D $  and $ P_{s}=P $ from the beginning of the proof yields the result.
\end{proof}

\section{An Application}\label{s:application}

Here we show an application of the first Riesz decomposition which is a Brelot type theorem.
\begin{theorem}\label{thm:Brelot}
	Let $H=H_{b,c,m}$ be subcritical and let $s$ be a non-negative superharmonic function on $X$. Then
	\[ Hs(x)\leq \inf_{y\in X}\frac{s(y)}{G(y,x)}=\frac{s(x)}{G(x,x)},\qquad x\in X. \]
	Moreover, there is equality 
	if and only if $s $ is a strictly positive multiple of $ G1_{o}$ and $ x=o $ or if $s=0$.
	%
\end{theorem}

\begin{remark}
	This theorem  has a continuous analogue, see e.g. \cite[Theorem~5.7.14]{HelmsNeu}, and goes back to Brelot, \cite{B44}. In the continuous case one even has equality. In contrast we show that in the discrete setting equality always fails to hold apart from two trivial cases. 
	The failure of the analogy between the discrete and the continuum setting stems from the fact that the discrete Green's function does not have a singularity at the diagonal.
\end{remark}

We need the following Harnack Principle.

\begin{lemma}[Harnack Principle, Lemma~4.6 in \cite{KePiPo1}]
	\label{lem:harnackprinc}
	Let $C>0$ and fix some $x\in X$. Assume that we have a sequence of positive superharmonic functions $(u_n)$ such that $C^{-1}\leq u_n(x) \leq C$. Then there exits a subsequence $(u_{n_k})$ that converges pointwise to a strictly positive superharmonic function $u$. 
\end{lemma}

Now, we prove the Brelot type theorem.

\begin{proof}[Proof of Theorem~\ref{thm:Brelot}]
	By the Riesz decomposition theorem, Theorem~\ref{thm:rieszdecompGHM}, we get for all $y\in X$
	\begin{align*}
	s(y)= GHs(y)+ \ghm_s(y) \geq G(y,x)Hs(x).
	\end{align*}
	Since $G(y,x)>0$ we get the desired inequality.

	\emph{Claim}: For all non-negative superharmonic functions $ s $ and $ x,y\in X $ we have
	\[\frac{s(y)}{G(x,y)}\geq \frac{s(x)}{G(x,x)}> 0.\]
	
	\emph{Proof of the claim.} Let us fix $x\in X$. Recall that for every finite set $K\subset X$ the inverse $(H^{K})^{-1}$ exists on $ \mathcal{C}(K) $. 
	Let $(K_n)$ be an increasing exhaustion of $X$ with finite sets with $ x\in K_{n} $, $ n\in\NN $. The goal is to apply the minimum principle, Theorem~\ref{thm:minprinc}, to 
	\[u = \frac{G(x,x)}{s(x)}s- (H^{K_n})^{-1}1_x\]
	for  $K=K_n\setminus\{x\}$, $n\in \NN$. Since $  (H^{K_n})^{-1}1_x$  is harmonic on  $ K_n\setminus\set{x}$, we have  $ Hu\ge0 $ on  $ K_n\setminus\set{x}$.
	Moreover, on $ X\setminus K_n$, we have $ u\geq 0 $ and, furthermore, $ u(x)=0 $ by definition of $ u $, i.e., $ u\ge0  $ on $ X\setminus K_{n} $. Hence, we can apply the minimum principle to $ u $ on $ K_{n}\setminus\{x\}$ and get that $u\geq 0$ on $X$. Since $(H^{K_n})^{-1}1_x\nearrow G1_x$ pointwise by \cite[Theorem~5.16]{KePiPo1}, we infer, for any $y\in X$
	\begin{align*}
	\frac{G(x,x)}{s(x)}s(y)- G(x,y)\ge0
	\end{align*}
and, therefore, the claim follows.

The claim proves the equality in the statement of the theorem.

Next, we turn to the characterization of equality. Clearly, we have equality whenever $ s $ is a strictly positive multiple of $ G1_{o} $ and $ x=o $, or $ s=0 $.

On the other hand, assume there is equality in $ x $. Then,
$$  G(x,x)Hs(x)= s(x)=GHs(x)+\ghm_{s}(x)  $$
by the Riesz decompostion, Theorem~\ref{thm:rieszdecompGHM}, and, therefore,
\begin{align*}
\sum_{z\in X\setminus\set{x}}G(x,z)Hs(z)+ \ghm_s(x)=0.
\end{align*}
Since all terms involved are non-negative by Theorem~\ref{thm:rieszdecompGHM}, we infer that they must be equal zero. 
If $ \ghm_{s} $ vanishes in $ x $ it vanishes everywhere by the the Harnack inequality, Proposition~\ref{thm:harnackineq}. This gives that $ s $ is a potential with non-negative charge which has to vanish everywhere outside of $ x $. This leaves the cases of  $s  $ being either a strictly positive  multiple of $ G1_{x} $ or $ s=0 $.
\end{proof}

\textbf{Acknowledgements.} The authors acknowledge the financial support of the DFG.\medskip

The paper is based on the first part of the master's thesis \cite{Masterarbeit}.

\newpage
\newcommand{\etalchar}[1]{$^{#1}$}

\end{document}